\documentclass[12pt,reqno]{amsart}
\usepackage{etoolbox}

\makeatletter
\patchcmd\maketitle
{\uppercasenonmath\shorttitle}
{}
{}{}
\patchcmd\maketitle
{\@nx\MakeUppercase{\the\toks@}}
{\the\toks@}
{}
{}{}
\patchcmd\@settitle{\uppercasenonmath\@title}{\Large}{}{}
\patchcmd\@setauthors
{\MakeUppercase{\authors}}
{\authors}
{}{}
\makeatother
\usepackage{amsmath,amssymb,amsthm, amsfonts, enumerate}
\usepackage{color}
\usepackage{url}
\usepackage{tikz-cd}
\usepackage{combelow}
\input{mathrsfs.sty}
\usepackage[utf8]{inputenc}
\usepackage[T1]{fontenc}
\newtheorem{theorem}{Theorem}[section]
\newtheorem{question}{Question}[section]

\newtheorem{corollary}{Corollary}[section]

\newtheorem{remark}{Remark}[section]

\numberwithin{equation}{section}

\newcommand\norm[1]{\left\lVert#1\right\rVert}
\newcommand\aps[1]{\left\lvert#1\right\rvert}
\newcommand\skal[2]{\left\langle #1,#2\right\rangle}
\newcommand\scal[2]{\langle #1,#2\rangle}

\usepackage[colorlinks=true]{hyperref}
\hypersetup{urlcolor=blue, citecolor=red , linkcolor= blue}
\usepackage{cite}
	
\begin{document}
		\title[Characterizations of positive operators via their powers]{Characterizations of positive operators via their powers}
		
		
		\author[H. Stankovi\'c]{Hranislav Stankovi\'c}
		\address{Faculty of Electronic Engineering, University of Ni\v s, Aleksandra Medvedeva 14, Ni\v s, Serbia
		}
		\email{\url{hranislav.stankovic@elfak.ni.ac.rs}}
		
		\date{\today}
		
		\maketitle
		
		\begin{abstract}
			In this paper, we present new characterizations of normal and positive operators in terms of their powers. Among other things, we show that if $T^2$ is normal, $\mathcal{W}(T^{2k+1})$ lies on one side of a line passing through the origin (possibly including some points on the line) for some $k\in\mathbb{N}$, and $\mathrm{asc\,}(T)= 1$ (or $\mathrm{dsc\,}(T)=1$), then $T$ must be normal. This complements the previous result due to Putnam \cite{Putnam57}. Furthermore, we prove that $T$ is normal (positive) if and only if $\mathrm{asc\,}(T)= 1$ and there exist coprime numbers $p,q\geq 2$ such that $T^p$ and $T^q$ are normal (positive). Finally, we also show that $T$ is positive if and only if $T^k$ is accretive for all $k\in\mathbb{N}$, which answers the question from \cite{Mortad22counterexamples} in the affirmative.
		\end{abstract}

		\section{Introduction}
		
		Let $\mathcal{H}$ be a complex Hilbert space and let $\mathfrak{B}(\mathcal{H})$ denote the set of all bounded linear operators on $\mathcal{H}$. For $T\in\mathfrak{B}(\mathcal{H})$, $\mathcal{N}(T)$  and $\mathcal{R}(T)$ stands for the null space and the range and  of $T$, respectively. If there exists $n\in\mathbb{N}$ such that $\mathcal{N}(T^n)=\mathcal{N}(T^{n+1})$ ($\mathcal{R}(T^n)=\mathcal{R}(T^{n+1})$), we say that operator $T$ has a finite ascent (descent). The smallest such $n$ is called the ascent (descent) of $T$ and it is denoted by $\mathrm{asc\,}(T)$ ($\mathrm{dsc\,}(T)$). \par 
		The adjoint of an operator $T\in\mathfrak{B}(\mathcal{H})$ will be denoted by $T^*$. Furthermore, the modulus of $T$ will be signified by $|T|=(T^*T)^{\frac{1}{2}}$, while $\mathrm{Re\,}{T}=\frac{T+T^*}{2}$ and $\mathrm{Im\,}{T}=\frac{T-T^*}{2i}$ will represent the real and the imaginary part of $T$, respectively. For the spectrum of $T$ we use the notation $\sigma(T)$, while the numerical range will be denote by $\mathcal{W}(T)$.
		
		An operator $T$ is said to be \emph{positive}, in notation $T\geq 0$, if $\skal{Tx}{x}\geq 0$ for all $x\in\mathcal{H}$, \emph{accretive} if $\mathrm{Re\,}{T}\geq 0$,  and \emph{self-adjoint (Hermitian)} if $T=T^*$.
		An operator $T$ is \textit{normal} if $T^*T=TT^*$. The study of normal operators has been distinctly successful. The main reason is definitely the Spectral Theorem that holds for such operators.
		Due to their importance in operator theory, as well as in quantum mechanics, many generalizations of the class have appeared over the decades, by weakening the commutativity condition. Some of the most important classes of such operators are the following:
		\begin{itemize}
			\item quasinormal operators: $T$ commutes with $T^*T$, i.e., $TT^*T=T^*T^2$;
			\item subnormal operators: there exist a Hilbert space $\mathcal{K}$, $\mathcal{K}\supseteq \mathcal{H}$, and a normal operator $N\in \mathfrak{B}(\mathcal{K})$ such that 
			$$
			N=\begin{bmatrix} T&\ast\\0&\ast\end{bmatrix}: \begin{pmatrix}
				\mathcal{H}\\
				\mathcal{H}^\perp
			\end{pmatrix}\to \begin{pmatrix}
				\mathcal{H}\\
				\mathcal{H}^\perp
			\end{pmatrix};
			$$
			\item hyponormal operators: $TT^*\leq T^*T$;
			\item paranormal operators: $\norm{Tx}^2\leq\norm{T^2x}$, for all $x\in\mathcal{H}$, $\norm{x}=1$;
		\end{itemize}
		It is well known that 
		\begin{align*}
			\text{positive}&\,\Rightarrow\,\text{self-adjoint}\,\Rightarrow\,\text{normal}\,\Rightarrow\,\text{quasinormal}\\
			&\,\Rightarrow\,\text{subnormal}\,\Rightarrow\,\text{hyponormal}\,\Rightarrow\,\text{paranormal},
		\end{align*}
		while self-adjoint $\nRightarrow$ accretive, and accretive $\nRightarrow$ normal, in general. For more details on these classes, as well as other generalizations of normal operators, see, for example, \cite{Furuta01}.
		\medskip 
		
		An interesting problem in operator theory involves determining conditions under which certain operators are normal. Many mathematicians have explored this topic, as noted in references \cite{Berberian70, Mortad10, Mortad19, MoslehianNabaviSales11} and the sources cited therein. A notable contribution by Stampfli \cite{Stampfli62} demonstrated that for a hyponormal operator \(T\), if \(T^n\) is normal for some \(n \in \mathbb{N}\), then \(T\) is normal. The  hyponormality condition can be further relaxed to paranormality, as seen in \cite{Ando72}. For the related problems for different operator classes, see \cite{CvetkovicIlicStankovic24, CurtoLeeYoon20, PietrzyckiStochel21, PietrzyckiStochel23, Stankovic23_factors, Stankovic23_roots, Stankovic24a}.
		
		Another method to tackle these problems involves examining the spectrum of an operator. In \cite{Putnam70}, the author proved that if the spectrum of a hyponormal operator has zero area, the operator must be normal. This result was later extended to more general classes of operators (see \cite{ChoItoh95, Tanahashi04, Xia83}).
		
		Conditions implying self-adjointness and positivity were also investigated by many authors and via many different approaches. See, for example, \cite{Berberian62, FongIstratescu79, FongTsui81, JeonKimTanahashi08, Kittaneh84, Mortad12}.
		
		\medskip
		
		The main goal of this paper is to present some conditions which would guarantee the positivity of an operator, by imposing some restrictions on their powers. Actually, this will allows us to give some new characterizations of positive operators. Moreover, we shall also examine the relations between positive and normal operators with accretive operators.
		
		The paper is organized as follows. In Section \ref{sec:preliminaries}, we present some known results, which will be frequently used in the paper and which will be crucial in the proofs. In Section \ref{sec:roots_of_normal}, we mainly deal with the square roots of normal operators, and present some conditions which imply positivity and normality of an operator, generalizing some classical Putnam's results on the topic, along the way. In Section \ref{sec:accretive_powers}, we show that an operator whose powers are accretive must be positive, and also explore whether  we can exclude some powers to be accretive, and still deduce the positivity of an operator. Finally, Section \ref{sec:conclusion} serves to collect all obtained characterizations of positive operators into one result, and also discuss possible further research.

		\section{Preliminaries}\label{sec:preliminaries}

		In this brief section, we shall recall some results and concepts which will be of importance later.

		Given a closed subspace $\mathcal{S}$ of $\mathcal{H}$, let $P_\mathcal{S}$ denote the orthogonal projection
		onto $\mathcal{S}$. For any $T\in\mathfrak{B}(\mathcal{H})$, the operator matrix decomposition of $T$ induced by $S$ is given by
		\begin{equation}\label{eq:matrix_decomposition}
			T=\begin{bmatrix}
				T_{11}&T_{12}\\
				T_{21}&T_{22}
			\end{bmatrix},
		\end{equation}
		where $T_{11}=P_STP_S|_{S}$, $T_{12}=P_ST(I-P_S)|_{S^\perp}$, $T_{21}=(I-P_S)TP_S|_{S}$ and $T_{22}=(I-P_S)T(I-P_S)|_{S^\perp}$.

		\begin{theorem}\cite{AndersonTrapp75}\label{thm:positive_matrix}
			Let $S$ be a closed subspace of $\mathcal{H}$ and $T\in\mathfrak{B}(\mathcal{H})$ have the matrix operator decomposition induced by $S$ and given by (\ref{eq:matrix_decomposition}). Then, $T$ is positive if and only if
			\begin{enumerate}[$(i)$]
				\item $T_{11}\geq 0$;
				\item $T_{21}=T^*_{12}$;
				\item $\mathcal{R}(T_{12})\subseteq\mathcal{R}(T_{11}^{1/2})$;
				\item $T_{22}=\left((T_{11}^{1/2})^{\dagger}T_{12}\right)^*(T_{11}^{1/2})^{\dagger}T_{12}+F$, where $F\geq 0$.
			\end{enumerate}
		\end{theorem}
		
		The crucial fact in many proofs will be the following representation theorem of the square roots of normal operators proved by Radjavi and Rosenthal in \cite{RadjaviRosenthal71}. We present it here in a slightly different form.
		\begin{theorem}\cite{RadjaviRosenthal71}\label{thm:normal_square}
			Let $T\in\mathfrak{B}(\mathcal{H})$. Operator $T$ is a square root of a normal operator if and only if 
			\begin{equation}\label{eq:normal_square_representation}
				T=\begin{bmatrix}
					A&0&0\\
					0&B&C\\
					0&0&-B
				\end{bmatrix},
			\end{equation}
			where $A, B$ are normal, $C\geq 0$, $C$ is one-to-one and $BC=CB$. Moreover, $B$ can be chosen so that $\sigma(B)$ lies in the closed upper half-plane and  the Hermitian part of $B$ is non-negative.
		\end{theorem}
		
		Finally, we also recall some celebrated results regarding normal operators and operator inequalities involving positive operators.
		\begin{theorem}
			[L\" owner-Heinz inequality {\cite{Heinz51, Lowner34}}]\label{thm:lowner_heinz_ineq} If $A,B\in\mathfrak{B}(\mathcal{H})$ are positive operators such that $B\leq A$ and $p\in[0,1]$, then $B^p\leq A^p$.
		\end{theorem}
		
		\begin{theorem}[Kato inequality \cite{Kato52}]\label{thm:kato_ineq}
			Let	$T\in\mathfrak{B}(\mathcal{H})$. Then
			\begin{equation*}
				\aps{\skal{Tx}{y}}^2\leq \skal{(T^*T)^\alpha x}{x}\skal{(TT^*)^{1-\alpha}y}{y},
			\end{equation*}
			for any $x,y\in\mathcal{H}$ and $\alpha\in[0,1]$.
		\end{theorem}
		
		\begin{theorem}[Fuglede-Putnam Theorem \cite{Fuglede50, Putnam51}]\label{thm:fuglede_putnam}\index{Fuglede-Putnam Theorem}
			Let $T\in\mathfrak{B}(\mathcal{H})$  and let $M$ and $N$ be two normal operators. Then
			\begin{equation*}
				TN=MT\,\Longleftrightarrow\, TN^*=B^*T.
			\end{equation*}
		\end{theorem}
		
		\begin{corollary}\cite{Fuglede50}\label{cor:fuglede_normal}
			If $M$ and $N$ are commuting normal operators,  then $MN$ is also normal.
		\end{corollary}
		\medskip 
		
		\section{Roots of normal operators}\label{sec:roots_of_normal}
		
		In a recent paper (\cite{Stankovic24}), as the corollaries of a more general examination, we have obtained the following results:
		
		\begin{corollary}\cite{Stankovic24}\label{cor:normal_n_geq 2}
			Let $T\in\mathfrak{B}(\mathcal{H})$. If $T^2$ is  normal and $\mathrm{Re\,}(T^{2k+1})\geq 0$ for some $k\in\mathbb{N}$, then $T^n$ is normal for all $n\geq 2$. 
		\end{corollary}

		\begin{corollary}\cite{Stankovic24}\label{cor:normal_injective}
			Let $T\in\mathfrak{B}(\mathcal{H})$. If $T$ is injective, $T^2$ is normal and $\mathrm{Re\,}(T^{2k+1})\geq 0$ for some $k\in\mathbb{N}$, then $T$ is normal. 
		\end{corollary}
		
		Note that under a suitable rotation, we can actually obtain the following more general result:
		
		\begin{corollary}\label{cor:normal_injective_gen}
			Let $T\in\mathfrak{B}(\mathcal{H})$. If $T^2$ is  normal and $\mathcal{W}(T^{2k+1})$ lies on one side of a line passing through the origin (possibly including some points on the line) for some $k\in\mathbb{N}$, then $T^n$ is normal for all $n\geq 2$. \par 
			If, in addition, $T$ is injective, then $T$ is also normal.
		\end{corollary}
		
		That the injectivity condition cannot be completely dropped, we can see by taking any non-zero operator $T$ such that $T^2=0$. However,  it can be significantly relaxed, as the following shows:
		
		\begin{theorem}\label{thm:asc_dsc}
			Let $T\in\mathfrak{B}(\mathcal{H})$ be such that
			\begin{enumerate}[$(i)$]
				\item $T^2$ is normal;
				\item $\mathcal{W}(T^{2k+1})$ lies on one side of a line passing through the origin (possibly including some points on the line) for some $k\in\mathbb{N}$;
				\item $\mathrm{asc\,}(T)= 1$ or $\mathrm{dsc\,}(T)=1$.
			\end{enumerate}
			Then $T$ is normal.
		\end{theorem}
		
		\begin{proof}
			From Corollary \ref{cor:normal_injective_gen}, we have that $T^n$ is normal for all $n\geq 2$. Specially, $T^2$ and $T^3$ are normal. Since $T^2$ and $T^3$ commute with $T$, Theorem \ref{thm:fuglede_putnam} implies that $T^2$  and $T^3$ also commute with $T^*$. 
			
			First consider the case when $\mathrm{dsc\,}(T)= 1$. By the previous discussion,
			\begin{equation*}
				T^*T T^2=T^*T^3=T^3T^*=TT^2T^*=TT^*T^2,
			\end{equation*}
			and thus, $(TT^*-T^*T)T^2=0$.  Since $\mathcal{R}(T)=\mathcal{R}(T^2)$ we have that $(TT^*-T^*T)T=0$, i.e., $TT^*T=T^*T^2$. In other words, $T$ is quasinormal. Using the fact that hyponormal (or even paranormal) $n$-th roots of normal operators are normal (see \cite{Ando72} or \cite{Stampfli62}), we conclude that $T$ is normal. 
			
			Now assume that $\mathrm{asc\,}(T)= 1$. From 
			\begin{equation*}
				T^2T^*T=T^*T^2T=T^*T^3=T^3T^*=T^2TT^*
			\end{equation*}
			we have that $T^2(TT^*-T^*T)=0$. Now, since $\mathcal{N}(T)=\mathcal{N}(T^2)$, we get that $T(TT^*-T^*T)=0$, i.e., $T^*$ is quasinormal. Furthermore, as $(T^*)^2$ is normal, from the same reason as it the previous case, we conclude that $T^*$ is normal, and thus $T$ is  normal.
		\end{proof}

		\begin{remark}\label{rem:hyponormal_roots}
			The implication ''$T$ quasinormal $\implies$ $T$ normal`` in the previous proof  can be also deduced directly from Theorem \ref{thm:normal_square}. Indeed, assume that $T$ is not normal. Using representation \eqref{eq:normal_square_representation}, a direct computation shows that 
			\begin{equation*}
				TT^*T=\begin{bmatrix}
					AA^*A&0&0\\
					0&BB^*B+C^2B&BB^*C+C^3-CB^*B\\
					0&-BCB&-BC^2-BB^*B
				\end{bmatrix},
			\end{equation*}
			and
			\begin{equation*}
				T^*T^2=\begin{bmatrix}
					A^*A^2&0&0\\
					0&B^*B^2&0\\
					0&CB^2&-B^*B^2
				\end{bmatrix}.
			\end{equation*}
			Specially,
			\begin{equation}\label{eq:2_2_entry}
				BB^*B+C^2B=B^*B^2
			\end{equation}
			and
			\begin{equation}\label{eq:2_3_entry}
				BB^*C+C^3-CB^*B=0.
			\end{equation}
			Normality of $B$ and \eqref{eq:2_2_entry} now imply that $C^2B=0$. Furthermore, since $C$ is one-to-one, it must be $B=0$. Combining this with \eqref{eq:2_3_entry} yields that $C^3=0$. However, this contradicts the fact that $C$ is one-to-one. Therefore, $T$ must be normal.\par 
			
		\end{remark}

		It is interesting to note that if we allow $k$ to be zero in previous theorem, then condition $(iii)$ is redundant. This fact is enclosed in the following result due to Putnam \cite{Putnam57}:
		
		\begin{theorem}\cite{Putnam57}\label{thm:putnam}
			Let $T\in\mathfrak{B}(\mathcal{H})$. If $\mathcal{W}(T)$ lies on one side of a line passing through the origin (possibly including some points on the line) and $T^2$ is normal, then  $T$ is also normal.
		\end{theorem}

		Thus, Theorem \ref{thm:asc_dsc} can be treated as a complement of the previous Putnam's result. Here, we also use an opportunity to provide an alternative of Theorem \ref{thm:putnam}.

		\begin{proof}[Proof of Theorem \ref{thm:putnam}]
			Without loss of generality we may assume that $\mathrm{Re\,}T\geq 0$. Otherwise, we can consider operator $S=e^{i\theta}T$, for an appropriate $\theta\in\mathbb{R}$. Assume to the contrary, that $T$ is not normal. Then the block $\begin{bmatrix}
				B&C\\ 0&-B
			\end{bmatrix}$ from \eqref{eq:normal_square_representation} is non-trivial. Using the fact that $\mathrm{Re\,}T\geq 0$, we easily get that
			\begin{equation*}
				\begin{bmatrix}
					B+B^* & C\\
					C&-(B+B^*)
				\end{bmatrix}\geq 0.
			\end{equation*}
			Theorem \ref{thm:positive_matrix} now implies that 
			\begin{enumerate}[$(i)$]
				\item $B+B^*\geq 0$;
				\item $\mathcal{R}(C)\subseteq\mathcal{R}((B+B^*)^{1/2})$;
				\item $-(B+B^*)\geq 0$.
			\end{enumerate}
			
			Clearly, conditions $(i)$ and $(iii)$ yield that $B+B^*=0$, while condition $(ii)$ further implies that $C=0$. However, this is not possible, since $C$ is one-to-one. This shows that $T$ must be normal.
		\end{proof}

		\begin{corollary}\label{cor:putnam_generalized_accr}
			Let $T\in\mathfrak{B}(\mathcal{H})$ be such that
			\begin{enumerate}[$(i)$]
				\item $T^2$ is normal;
				\item $\mathrm{Re\,}T^{2k+1}\geq 0$ for some $k\in\mathbb{N}$;
				\item $\mathrm{asc\,}(T)= 1$ or $\mathrm{dsc\,}(T)=1$.
			\end{enumerate}
			Then $T$ is normal.
		\end{corollary}

		As a direct consequence of the previous result, we have the following:
		
		\begin{corollary}\label{cor:k_accretive_positive}
			Let $T\in\mathfrak{B}(\mathcal{H})$ be such that
			\begin{enumerate}[$(i)$]
				\item $T^2\geq 0$;
				\item $\mathrm{Re\,}T^{2k+1}\geq 0$ for some $k\in\mathbb{N}$;
				\item $\mathrm{asc\,}(T)= 1$ or $\mathrm{dsc\,}(T)=1$.
			\end{enumerate}
			Then $T\geq 0$.
		\end{corollary}

		\begin{proof}
			By Corollary \ref{cor:putnam_generalized_accr}, it follows that $T$ is normal. Now let $\lambda\in\sigma(T)$. By the continuous functional calculus, we have that $\lambda^2\geq 0$ and $\mathrm{Re\,}\lambda^{2k+1}\geq 0$, which yields $\lambda\geq 0$. Therefore, $\sigma(T)\subseteq[0,\infty)$, which directly implies that $T$ is a positive operator.
		\end{proof}
		
		\begin{remark}\label{rem:zero_case_positive}
			The case $k=0$ of the previous corollary was already established in \cite{Putnam58}, and again, it does not require condition $(iii)$.
		\end{remark}
		Also, the natural question arises whether we can replace a pair $(2,2k+1)$ with a pair $(p,q)$, where $p, q\geq 2$ are coprime numbers. The answer here is negative. Indeed, let $T=e^{i\frac{\pi}{3}}I$. Then $T$ is bijective, $T^6=I\geq 0$ and $\mathrm{Re\,}T^7=\mathrm{Re\,}T=\cos\frac{\pi}{3}I\geq 0$, while $T$ is not positive. Note that the similar questions have been already considered. In \cite{DehimiMortad23}, it was shown that if $T$ is invertible and $T^p$ and $T^q$ are normal for some coprime numbers $p,q\geq 2$, then $T$ is also normal. In the following theorem, we generalize this result to the point where the other direction holds, as well.
		
		\begin{theorem}\label{thm:comprime_normal}
			Let $T\in\mathfrak{B}(\mathcal{H})$. Then the following conditions are equivalent:
			\begin{enumerate}[$(i)$]
				\item $T$ is normal;
				\item $\mathrm{asc\,}(T)= 1$ and there exist coprime numbers $p,q\geq 2$ such that $T^p$ and $T^q$ are normal.
			\end{enumerate}
		\end{theorem}

		\begin{proof} From the obvious reasons, we only prove $(ii)\Rightarrow(i)$. Since $p$ and $q$ are coprime, by B\' ezout's theorem, there exist $k, l\in\mathbb{Z}$ such that $1=kp+lq$. Let $n:=|k|p+|l|q$. Then, $n+1=(|k|+k)p+(|l|+l)q$, and so
			\begin{equation*}
				T^n=(T^p)^{|k|}(T^q)^{|l|},\quad 
				\text{ and }\quad T^{n+1}=(T^p)^{|k|+k}(T^q)^{|l|+l}.
			\end{equation*}
			Since $T^p$ and $T^q$ are commuting normal operators, Corollary \ref{cor:fuglede_normal} implies that $T^n$ and $T^{n+1}$ are also normal. Using the similar technique as in the proof of Theorem \ref{thm:asc_dsc}, we can obtain that $T$ is normal.
		\end{proof}

		\begin{corollary}\label{cor:prime_numbers_positive}
			Let $T\in\mathfrak{B}(\mathcal{H})$ be such that $\mathrm{asc\,}(T)= 1$. If there exist coprime numbers $p,q\geq 2$ such that $T^p\geq 0$ and $T^q\geq 0$, then $T\geq 0$.
		\end{corollary}
		
		\begin{proof}
			By Theorem \ref{thm:comprime_normal}, it follows that $T$ is normal. Now let $\lambda=re^{i\theta}\in\sigma(T)$. By the Spectral Mapping Theorem, we have that $\lambda^p\geq 0$ and $\lambda^q\geq 0$. Therefore, there exist $k,l\in\mathbb{Z}$ such that
			\begin{equation*}
				p\theta=2k\pi\quad 
				\text{ and }\quad q\theta=2l\pi.
			\end{equation*}
			From here, $pl=qk$, and since $(p,q)=1$, we have that $p|k$. This shows that $\theta=2\frac{k}{p}\pi$, $\frac{k}{p}\in\mathbb{Z}$. We conclude that $\sigma(T)\subseteq[0,\infty)$, and thus $T\geq 0$.
		\end{proof}

		\medskip  
		\section{Accretivity of powers}\label{sec:accretive_powers}
		We continue the examination by considering the recent question posed by the author in \cite{Mortad22counterexamples}:
		\begin{question}\cite[p. 288]{Mortad22counterexamples}\label{question_mortad}
			Let $T\in\mathfrak{B}(\mathcal{H})$. If $\mathrm{Re\,}T^{k}\geq 0$ for each $k\in\mathbb{N}$, does it follow that $T\geq 0$?
		\end{question}
		
		Under the additional assumption that $T$ is normal, we can easily get an affirmative answer to the previous question. Indeed, if $\lambda\in\sigma(T)$, then using the continuous functional calculus, we have that $\mathrm{Re\,}\lambda^k\geq 0$ for all $k\in\mathbb{N}$, which certainly implies that $\lambda\geq 0$. Thus, $\sigma(T)\subseteq[0,\infty)$ and so $T$ must be positive.\par 
		The normality can be further replaced with the belonging to the much wider class of hyponormal operators. Let $k\in\mathbb{N}$ be arbitrary. Since $\mathrm{Re\,}T^{k}\geq 0$, we have that for all $x\in\mathcal{H}$, $\norm{x}=1$, $$\mathrm{Re}\skal{T^{k}x}{x}=\skal{\mathrm{Re\,}T^{k}x}{x}\geq 0.$$
		Thus, using, for example, \cite[Ch. 1, Proposition 2.1]{GauWu21}, 
		$$\sigma(T^k)\subseteq\overline{\mathcal{W}(T^k)}\subseteq \{z\in\mathbb{C}:\,\mathrm{Re\,}z\geq 0\},$$
		and the previous inclusion holds for all $k\in\mathbb{N}$. Now let $\lambda\in\sigma(T)$. Again, we can easily check that $\lambda\geq 0$ and consequently $\sigma(T)\subseteq[0,\infty)$. Since a hyponormal operator with the positive spectrum is positive (see \cite{Putnam65} or \cite{Stampfli65}), once again we conclude that $T\geq 0$.\par 
		
		Can we further relax the hyponormality condition? The answer is yes. In fact, in the next theorem, we provide an affirmative answer to the Question \ref{question_mortad} without any additional assumption, and even without requiring all the powers of $T$ to be accretive.

		\begin{theorem}\label{thm:all_accretive_positive}
			Let $T\in\mathfrak{B}(\mathcal{H})$. If $\mathrm{Re\,}T^{2^k}\geq 0$ for each $k\in\mathbb{N}_0$, then $T\geq 0$.
		\end{theorem}

		\begin{proof}
			Let $A:=\mathrm{Re\,}T$ and $B:=\mathrm{Im\,}T$. Then, 
			\begin{align*}
				T^2=(A+iB)^2=A^2-B^2 + i(AB+BA),
			\end{align*}
			i.e., $\mathrm{Re\,}T^2=A^2-B^2$ and $\mathrm{Im\,}T^2=AB+BA$. By assumption, $\mathrm{Re\,}T^2\geq 0$, and thus we have that $A^2\geq B^2$. Since $B$ is self-adjoint, it follows that $A^2\geq |B|^2$. From $A\geq 0$, and Theorem \ref{thm:lowner_heinz_ineq}, we obtain that $A\geq |B|$.\par  Now let $x\in\mathcal{H}$ be arbitrary. Using Theorem \ref{thm:kato_ineq}, the self-adjointness of $B$ implies that
			\begin{equation*}
				\aps{\skal{Bx}{x}}^2\leq \scal{(B^2)^{\frac{1}{2}} x}{x}^2,
			\end{equation*}
			i.e., \begin{equation}\label{eq:b_|b|}
				\aps{\skal{Bx}{x}}\leq \skal{|B| x}{x}.
			\end{equation}
			Combining this with $A\geq |B|$, we obtain
			\begin{equation*}
				\aps{\skal{Bx}{x}}\leq \skal{|B| x}{x}\leq\skal{A x}{x}.
			\end{equation*}
			From here, it follows that
			\begin{equation}\label{eq:-a_b_a}
				-A\leq B\leq A.
			\end{equation}
			Next, let us show that 
			\begin{equation}\label{eq:-a_cot_b_a}
				\mathrm{Re\,}T\geq 0, \mathrm{Re\,}T^{2}\geq 0,\ldots, \mathrm{Re\,}T^{2^n}\geq 0\quad\implies\quad	-A\leq \cot \left(\frac{\pi}{2^{n+1}}\right) B\leq A
			\end{equation}
			for all $n\in\mathbb{N}$.  The statements is true for $n=1$, by \eqref{eq:-a_b_a}. Now assume that is it also true for some $n\in\mathbb{N}$ and let 
			\begin{equation*}
				\mathrm{Re\,}T\geq 0, \mathrm{Re\,}T^{2}\geq 0,\ldots, \mathrm{Re\,}T^{2^n}\geq 0, \mathrm{Re\,}T^{2^{n+1}}\geq 0.
			\end{equation*}
			By applying the induction hypothesis to an operator $T^2$, we get that $$-\mathrm{Re\,}T^2\leq \cot \left(\frac{\pi}{2^{n+1}}\right)\mathrm{Im\,}T^2\leq \mathrm{Re\,}T^2,$$
			i.e.,
			\begin{equation*}
				-(A^2-B^2)\leq \cot \left(\frac{\pi}{2^{n+1}}\right)(AB+BA)\leq A^2-B^2.
			\end{equation*}
			From $\cot \left(\frac{\pi}{2^{n+1}}\right)(AB+BA)\leq A^2-B^2$, we have
			\begin{equation*}
				B^2\leq A^2-\cot \left(\frac{\pi}{2^{n+1}}\right)AB-\cot \left(\frac{\pi}{2^{n+1}}\right)BA,
			\end{equation*}
			and so
			\begin{equation*}
				\left(1+\cot^2 \left(\frac{\pi}{2^{n+1}}\right)\right)|B|^2\leq \left(A-\cot \left(\frac{\pi}{2^{n+1}}\right) B\right)^2.
			\end{equation*}
			Again, Theorem \ref{thm:lowner_heinz_ineq} and \eqref{eq:b_|b|} implies that
			\begin{equation*}
				\sqrt{1+\cot^2 \left(\frac{\pi}{2^{n+1}}\right)}B\leq \sqrt{1+\cot^2 \left(\frac{\pi}{2^{n+1}}\right)}|B|\leq A-\cot \left(\frac{\pi}{2^{n+1}}\right) B,
			\end{equation*}
			i.e.,
			\begin{equation}\label{eq:nonsimplified_expression}
				\left(\cot \left(\frac{\pi}{2^{n+1}}\right)+\sqrt{1+\cot^2 \left(\frac{\pi}{2^{n+1}}\right)}\right) B\leq A.
			\end{equation}
			Furthermore, we can easily check that for any $x\in\left(0,\pi\right)$,
			\begin{equation*}
				\cot \frac{x}{2}=\cot x+\sqrt{1+\cot^2 x}.
			\end{equation*}
			This, together with \eqref{eq:nonsimplified_expression} yields
			\begin{equation*}
				\cot \left(\frac{\pi}{2^{n+2}}\right)B\leq A.
			\end{equation*}
			Similarly, we can check that  $-(A^2-B^2)\leq \cot \left(\frac{\pi}{2^{n+1}}\right)(AB+BA)$ implies
			\begin{equation*}
				-A\leq \cot \left(\frac{\pi}{2^{n+2}}\right)B.
			\end{equation*}
			Therefore, we have shown that
			\begin{equation*}
				-A\leq \cot \left(\frac{\pi}{2^{n+2}}\right)B\leq A,
			\end{equation*}
			which demonstrates that \eqref{eq:-a_cot_b_a} holds for all $n\in\mathbb{N}$.\par 
			Using the initial assumption, we now have that for all $n\in\mathbb{N}$, 
			\begin{equation*}
				-\tan \left(\frac{\pi}{2^{n+1}}\right)A\leq B\leq \tan \left(\frac{\pi}{2^{n+1}}\right) A.
			\end{equation*}
			By letting $n\to\infty$, we conclude that $B=0$. Finally, this implies that $T=A\geq 0$.
		\end{proof}

		\begin{corollary}
			Let $T\in\mathfrak{B}(\mathcal{H})$. If $\mathrm{Re\,}T^{k}\geq 0$ for each $k\in\mathbb{N}$, then $T\geq 0$.
		\end{corollary}

		\begin{corollary}
			Let $T\in\mathfrak{B}(\mathcal{H})$. If $\mathrm{asc\,}(T)= 1$ and there exists $k_0\in\mathbb{N}$ such that $\mathrm{Re\,}T^{k}\geq 0$ for all $k\geq k_0$, then $T\geq 0$.
		\end{corollary}
		
		\begin{proof}
			Let $p, q\geq k_0$ be two coprime numbers. Then, for all $n\in\mathbb{N}$, $\mathrm{Re\,}(T^{p})^n=\mathrm{Re\,}T^{np}\geq 0$, and similarly, $\mathrm{Re\,}(T^{q})^n\geq 0$. By Theorem \ref{thm:all_accretive_positive}, we have that $T^p\geq 0$ and $T^q\geq 0$. The rest of the proof now follows from Corollary \ref{cor:prime_numbers_positive}.
		\end{proof}

		From the proof of Theorem \ref{thm:all_accretive_positive}, more specifically, from \eqref{eq:-a_cot_b_a}, we also have the following:
		
		\begin{theorem}
			Let $T\in\mathfrak{B}(\mathcal{H})$ and $n\in\mathbb{N}$. If $\mathrm{Re\,}T^{2^k}\geq 0$ for $k\in\{0, 1,\ldots, n\}$, then
			\begin{equation*}
				\mathcal{W}(T)\subseteq\left\{z\in\mathbb{C}:\, |\arg z|\leq \frac{\pi}{2^{n+1}}\right\}.
			\end{equation*}
		\end{theorem}
		
		\begin{proof}
			Let $x\in\mathcal{H}$, $\norm{x}=1$. From $-A\leq \cot \left(\frac{\pi}{2^{n+1}}\right) B\leq A$ we can easily check that
			\begin{equation*}
				\left|\mathrm{Im}\skal{Tx}{x}\right|\leq \tan \left(\frac{\pi}{2^{n+1}}\right)\mathrm{Re}\skal{Tx}{x},
			\end{equation*}
			from where the conclusion directly follows.
		\end{proof}

		\medskip
		\section{Conclusion}\label{sec:conclusion}
		In this paper, we have obtained several new characterizations of positive operators, mostly in terms of the accretivity of their powers. We collect the results into the following theorem:
		\begin{theorem}
			Let $T\in\mathfrak{B}(\mathcal{H})$. The following conditions are equivalent:
			\begin{enumerate}[$(i)$]
				\item $T\geq 0$;
				\item $T^2\geq 0$ and $\mathrm{Re\,}T\geq 0$;
				\item $\mathrm{asc\,}(T)= 1$, $T^2\geq 0$, and $\mathrm{Re\,}T^{2k+1}\geq 0$ for some $k\in\mathbb{N}$;
				\item $\mathrm{asc\,}(T)= 1$, and there exist coprime numbers $p,q\geq 2$ such that $T^p\geq 0$ and $T^q\geq 0$;
				\item $\mathrm{Re\,}T^{2^k}\geq 0$ for each $k\in\mathbb{N}_0$;
				\item $\mathrm{Re\,}T^{k}\geq 0$ for each $k\in\mathbb{N}$;
				\item $\mathrm{asc\,}(T)= 1$, and there exists $k_0\in\mathbb{N}$ such that $\mathrm{Re\,}T^{k}\geq 0$ for all $k\geq k_0$.
			\end{enumerate}
		\end{theorem}

		Finally, by looking at Corollary \ref{cor:putnam_generalized_accr} and Corollary \ref{thm:comprime_normal}, it is natural to ask the following question:
		
		\begin{question}
			Let $T\in\mathfrak{B}(\mathcal{H})$ be such that $\mathrm{asc\,}(T)= 1$. If there exist coprime numbers $p,q\geq 2$ such that $T^p$ is normal  and $\mathrm{Re\,}T^q\geq 0$, does it follow that $T$ is normal?
		\end{question}
		
		\medskip 
		\section*{Declarations}
		\noindent{\bf{Funding}}\\
		This work has been supported by the Ministry of Science, Technological Development and Innovation of the Republic of Serbia [Grant Number: 451-03-137/2025-03/ 200102].		
		
		\vspace{0.5cm}
		
		\noindent{\bf{Availability of data and materials}}\\
		\noindent No data were used to support this study.
		\vspace{0.5cm}\\
		\noindent{\bf{Competing interests}}\\
		\noindent The author declares that he has no competing interests.
		\vspace{0.5cm}
		
		
		\vspace{0.5cm}
		
		
		
	\end{document}